\documentclass{article}
\usepackage{amsmath,amssymb,amsthm,stmaryrd,makeidx}
\usepackage[all]{xy}

\DeclareMathOperator{\enm}{End}
\DeclareMathOperator{\ext}{Ext}

\DeclareMathOperator{\id}{id}

\DeclareMathOperator{\hmm}{Hom}
\DeclareMathOperator{\mor}{Mor}

\DeclareMathOperator{\spec}{Spec}
\DeclareMathOperator{\aspec}{aSpec}

\DeclareMathOperator{\simp}{Simp}

\newcommand{\cat}[1]{\mathbf{#1}}

\newcommand{\grmcat}[1]

\input xypic

\newtheorem{proposition}{Proposition}
\newtheorem{theorem}{Theorem}
\newtheorem{lemma}{Lemma}
\newtheorem{corollary}{Corollary}
\newtheorem{definition}{Definition}

\newtheorem{example}{Example}

\makeindex

\begin{document}
\author{Arvid Siqveland}
\title{Schemes of Associative algebras}

\maketitle

\begin{abstract} 
We give a definition of associative schemes, schemes of associative rings, over a field $k,$ using the definition of  completion of an associative $k$-algebra in a finite set of simple modules. We start by giving a weaker but sufficient definition of ordinary schemes of commutative rings which can be generalized to associative rings.
\end{abstract}

\section{Algebraic Geometry}

\begin{definition} Let $S$ be a set and let $\cat C$ be a small category. An object $U\in\cat C$ is called a fine moduli for $S$ if there exists an object $I\in\cat C$ such that $$S\simeq\mor_{\cat C}(I,U).$$
\end{definition}

Thus the elements in the set $S$ have a structure as a set of morphisms in a category.

\begin{example}[Algebraic Varieties] Let $A=k[x_1,\dots,x_n]/\mathfrak a$ be a domain, $k$ algebraically closed. Let $V=Z(\mathfrak a)$ be the set of maximal ideals in $A.$ Then we know that $A$ is a fine moduli for $V$ (in the category $\cat{Alg}_k^{op}$) because $$\hmm_k(A,k)\simeq V.$$
\end{example}

\begin{example}[Projective Varieties] Let $S$ be a graded $k$-algebra, and let $\mathfrak p$ be a homogeneous prime ideal. Let $V=Z(\mathfrak p)$ be the set of homogeneous maximal ideals. Then $\hmm^0_k(S,k)$ is not in bijective correspondence with $V$, even when considering graded homomorphisms. This is because the grading can be twisted. Considering graded rings and graded homomorphisms is not a categorical treatment of projective varieties.
\end{example}

It is well known that projective varieties can be covered by their affine open subschemes: The French (?) idea: If we cannot find a fine moduli for a set $S$ in a category $\cat C,$ then it can be possible to extend the category to $\tilde{\cat C},$ i.e. such that $\cat C$ is a subcategory, and such that $U\in\tilde{\cat C}$ is a fine moduli for $S.$ This is one advantage with schemes, if not the initial intention.

\section{The Category of Schemes Over $k$}

Let $A$ be a $k$-algebra, let $S=\operatorname{Max} A$ be the set of maximal ideals $\mathfrak m\subset A.$ This is a nice set, but by the above, we are looking for a (contravariant) functor $$F:\cat{Alg}_k\rightarrow\cat{Sets}$$ such that $F(k)=S.$ Given a homomorphism $f:A\rightarrow B,$ if $\mathfrak m\subset B$ is maximal (or prime), then $f^{-1}(\mathfrak m)$ is not necessarily maximal, but always prime. This means that we have a contravariant functor $\spec:\cat{Alg}_k\rightarrow\cat{Sets}$ defined by $\spec A=\{\mathfrak p\subset A|\ \mathfrak p\text{ is prime }\}$ fulfilling the condition.
Let $f\in A$ and define $D(f)=\{\mathfrak p\subset A|f\notin\mathfrak p\}\subseteq\spec A.$ Then because, for $f,g\in A,\ D(f)\cap D(g)=D(fg),$ and also $D(0)=\spec A,\ D(1)=\emptyset,$ the family $\{D(f)|f\in A\}$ is a basis for a topology (the Zariski topology) on $\spec A.$

\begin{lemma}\label{sheaflemma} Let $X$ be a topological space and let $F$ be a presheaf on $X,$ that is a contravariant functor $F:\operatorname{Top}X\rightarrow\cat{Rings}.$ Then the functor $$\mathcal F(U)=\underset{\underset{V\varsubsetneqq U}\leftarrow}\lim\ F(V)$$ is a sheaf on $X.$
\end{lemma}

\begin{proof} We can prove the conditions iii) and iv)  in Hartshorne \cite{HH77}, but it follows by, or can be taken as a, definition.
\end{proof}

\begin{definition} The sheaf in Lemma \ref{sheaflemma} is called the sheaf $\mathcal F$ associated to the presheaf $F.$
\end{definition}

Define a presheaf $O_X$ on $X=\spec A$ by $O_X(U)=\underset{\mathfrak p\in U}\prod A_{\mathfrak p}.$

\begin{definition} An affine scheme is the pair $U=(\spec A,\mathcal O_{\spec A}).$ A morphism of affine schemes
$f:(\spec A,\mathcal O_{\spec A})\rightarrow(\spec B,\mathcal O_{\spec B})$
 is a morphism induced by a ring homomorphism $f:B\rightarrow A.$
\end{definition}

\begin{definition}\label{schemedef} A scheme is a topological space $X$ with a sheaf of rings $\mathcal O_X,$ such that $X$ has an open cover of open subsets $U$ such that $$(U,\mathcal O_X|_U)\simeq (\spec A,\mathcal O_{\spec A})$$ for some ring $A.$
\end{definition}

\begin{theorem} The scheme $\mathbb P^n_k$ is a fine moduli for $\mathbb A^n_k/k^\ast.$ 
\end{theorem}

\begin{proof} Follows from the standard construction in Hartshorne, \cite{HH77}.
\end{proof}

We understand now that if $G$ is a nonabelian group acting on a variety $X,$ we cannot expect that there exists a fine moduli in the category of schemes. This is simply because the orbits corresponds to modules over noncommutative algebras $k[G],$ the skew group-algebra over $G.$

The next section looks for a definition of schemes that is equivalent to the previous definition, at least for noetherian domains, and that can be generalized to noncommutative algebras. The main component that cannot be generalized directly, is the localization $A_{\mathfrak p}$ in a prime $\mathfrak p.$

\section{Integral Noetherian Schemes}

\begin{definition} 
Let $A$ be a ring and $\mathfrak p\subset A$ a prime ideal. Consider the morphism $\rho:A\rightarrow\hat A_{\mathfrak p}.$
We define the Hausdorff localization of a ring $A$ in a prime $\mathfrak p\subset A$ as the subring $H^A_{\mathfrak p}\subset\hat A_{\mathfrak p}$ generated by $\rho(A)$ and the set 
$S=\{\rho^{-1}(a)|\text{ when }\rho(a)\text{ is invertible for }a\in A\}.$
\end{definition}

\begin{lemma} When $A$ is a noetherian integral domain, $H^A_{\mathfrak p}\simeq A_{\mathfrak p}.$
\end{lemma}

\begin{proof} We see that $H^A_{\mathfrak p}\simeq A_{\mathfrak p}/\underset{n\in\mathbb N}\cap\mathfrak m^n,$ where $\mathfrak m$ denotes the maximal ideal in the local ring $A_\mathfrak p.$ Then it follows from Krull's Theorem that $\underset{n\in\mathbb N}\cap\mathfrak m^n$ is the set of elements in $A$ annihilated by $1+x, x\in\mathfrak m.$ The only element annihilated by $1+x$ is $0$ in a noetherian domain, and so the result follows.
\end{proof}

\begin{definition} We will call a scheme $X$ locally Hausdorff if all local rings $A_\mathfrak p\simeq H^A_\mathfrak p.$
\end{definition}

\begin{corollary} A noetherian integral scheme is locally Hausdorff.
\end{corollary}

In the definition of schemes, Definition \ref{schemedef}, if we replace all local rings $A_\mathfrak p$ with $\hat A_\mathfrak p$ we get a formal scheme, which is of great interest because it gives a fine moduli in the category of formal schemes. The only problem is that the category of formal schemes is not an extension of schemes, and so gives a coarser moduli space.

We define a scheme as in Definition \ref{schemedef}, but we replace all local rings $A_\mathfrak p$ with $H^A_\mathfrak p.$ Then we get a different sheaf $\mathcal O_{\spec A}=\mathcal O_{\spec A^H}$ where $A^H=A/\mathfrak h,$ $\mathfrak h=\underset{\mathfrak m,n}\cap\mathfrak m^n,$ the intersection of all powers of all maximal ideals in $A.$
Thus we get an ordinary scheme which is locally Hausdorff, convenient for the replacement of Cauchy sequences in algebra.
 
We now have all we need to define schemes of associative rings.

\section{Associative Schemes over $k$}
In this section we will assume for simplicity that all rings are $k$-algebras with $k$ a field. All modules will be right modules unless otherwise stated, the word ideal means two-sided ideal.

\begin{definition} An associative ring $R$ with exactly $r$ maximal ideals, is called an $r$-pointed ring. When $r=1$ we call it a pointed ring.
\end{definition}

\subsection{Augmented $k$-algebras}
Let $k$ be a field and $r\in\mathbb N_+.$ Let $\{Q_i\}_{i=1}^r$ be a family of pointed $k$-algebras containing $k.$ A $k$-algebra $H$ is augmented over $\oplus_{i=1}^r Q_i$ if it  fits in the diagram 
$$\xymatrix{k^r\ar[r]^\rho\ar[dr]_{\iota}&H\ar@{>>}[d]^{\pi_H}\ar[dr]^{\tilde{\pi}_H}&\\&\oplus_{i=1}^r Q_i\ar@{>>}[r]&\oplus_{i=1}^r Q_i/\mathfrak m_i}$$ where $\mathfrak m_i\subset Q_i$ is the unique maximal ideal.

Let $\mathfrak m=\ker\tilde\pi_H$ which is a two-sided ideal, not necessarily maximal because $k^r\subseteq H/\mathfrak m$ is not a division ring for $r\geq 2.$

\begin{definition} The $\mathfrak m$-adic completion of $H$ is defined as $$\hat H=\underset{\underset{n\geq 0}\leftarrow}\lim H/\mathfrak m^n.$$ Then $\hat H$ fits in the diagram $$\xymatrix{k^r\ar[r]^\rho\ar[dr]_{\id}&\hat H\ar@{>>}[d]^\pi\\&\oplus_{i=1}^r Q_i,}$$ and by definition $\hat{\hat H}\simeq\hat H.$ Any $k$-algebra $S$ augmented over some $\oplus_{i=1}^r Q_i$ with the property that $\hat S\simeq S$ is called complete.
\end{definition}

In \cite{S243} we prove the following Lemma:
\begin{lemma}\label{unitlemma}
Let $\hat H$ be a complete algebra augmented over $k^r.$ Let $\pi_i$ be the composition of the augmentation with the $i$'th projection, and let $\mathfrak m_i=\ker\pi_i.$ If $x\notin\cup_{i=1}^r\mathfrak m_i,$ then $x$ is a unit.
\end{lemma}

From Lemma \ref{unitlemma} it follows (\cite{S243}) that a complete $k$-algebra $\hat H$ augmented over $k^r$ is $r$-pointed, and in particular that $r=1$ implies that $\hat H$  is local.

Let $V_{ij}=k,\ 1\leq i,j\leq r.$  The $r\times r$-matrix $V=(V_{ij})$ is a $k^r$-algebra, and we can consider the tensor algebra $T(V)=\oplus_{i\geq 0}V^{{\otimes_{k^r}}^i}.$ 
Then $$T(V)=\left(\begin{matrix}k[t_{11}]&\cdots&t_{1r}\\
\vdots&\dots&\vdots\\
t_{r1}&\cdots&k[t_{rr}]\end{matrix}\right)$$ is a $k$-algebra, augmented over $k^r$ by $\pi:T(V)\rightarrow k^r.$  We put $T=T(V)/\ker^2\pi$ and name it the $r\times r$ tangent algebra.
\subsection{Completions of Algebras}

Let $\phi:R\rightarrow S$ be a morphism between two $r$-pointed $k^r$-algebras, and let $M$ be a $k^r$-module. Then we have an algebra homomorphism $\Pi(\phi):\enm_{R}(R\otimes_{k^r}M)\rightarrow\enm_S(S\otimes_{k^r}M)$ given by the identity $S\otimes_R(R\otimes_{k^r}M)\simeq S\otimes_{k^r}M.$

\begin{definition}\label{precompdef} Let $A$ be a $k$-algebra and $M=\oplus_{i=1}^rM_i$ a sum of right $A$-modules, $r\geq 1.$ A complete $k$-algebra $\hat H^A_M$ augmented over $k^r$ is a pre-completion of $A$ in $M$ if the following (versal) condition holds: There exists a homomorphism $\hat\rho^A_M:A\rightarrow\enm_{\hat H^A_M}(\hat H^A_M\otimes_{k^r}M)$ fitting in the diagram $$\xymatrix{A\ar[d]_\phi\ar[r]^-{\hat\rho^A_M}\ar[dr]_(0,3){\oplus\eta_{M_i}}&\enm_{\hat H^A_M}(\hat H^A_M\otimes_{k^r}M)\ar[d]^{\Pi(\pi_{\hat H})}\ar@{-->}[dl]_{\Pi(\psi)}\\\enm_S(S\otimes_{k^r}M)\ar[r]_{\Pi(\pi_S)}&\oplus_{i=1}^r\enm_k(M_i)}$$ and such that if $S$ is any other such algebra with a homomorphism $$\phi:A\rightarrow\enm_{S}(S\otimes_{k^r}M)$$ fitting in the diagram, then there exists a morphism $\psi:\hat H^A_M\rightarrow S,$ unique in the case  $S\cong T,$ commuting in the above diagram.
\end{definition}

In \cite{S243} we prove the following theorem and proposition.

\begin{theorem}\label{deftheorem} Let $A$ be a finitely generated associative $k$-algebra. Let $M=\oplus_{i=1}^rM_i$ be a sum of right $A$-modules such that for each $1\leq i,j\leq r$ we have that $\dim_k\ext^1_A(M_i,M_j)<\infty.$ Then the pre-completion of $A$ in $M$ exists and is unique up to isomorphism.
\end{theorem}

\begin{proposition} Assume that $A$ is a $k$-algebra augmented over $k^r$ and let $M_i=A/\ker(p_i\circ\pi_A)\simeq k,$ where $p_i:k^r\rightarrow k$ is the $i$'th projection. Then the $\mathfrak m$-adic completion $\hat A$ of $A$ is a pre-completion of $A$ in $M=\oplus_{i=1}^r M_i.$
\end{proposition}

\begin{definition}
Let $A$ be a $k$-algebra and  $M=\oplus_{i=1}^rM_i$ a  sum of $r$ right $A$-modules. If $A$ has a pre-completion $\hat H^A_M$ we define the completion of $A$ in $M$ as $$\hat O^A_M=\enm_{\hat H^A_M}(\hat H^A_M\otimes_{k^r}M).$$
\end{definition}

By definition, the completion of $A$ in $M,$ if it exists, comes with a homomorphism $\rho^A_M:A\rightarrow\hat O^A_M$ commuting in the diagram $$\xymatrix{A\ar[r]^{\rho^A_M}\ar[dr]_{\oplus_{i=1}^r\eta_{M_i}}&\hat O^A_M\ar@{>>}[d]^\pi\\&\oplus_{i=1}^r\enm_k(M_i).}$$ 
Also, $\hat O^A_M$ is augmented over $\oplus_{i=1}^r\enm_k(M_i)$ where each $\enm_k(M_i)$ is a pointed $k$-algebra.

\begin{proposition} The completion of $A$ in $M$ if it exists, is complete.
\end{proposition}

\begin{proposition} The completion $\hat O^A_M$ of $A$ in a sum of right  modules $M=\oplus_{i=1}^rM_i$ has $\tilde M=\{M_i\}_{i=1}^r$ as its set of simple right and left modules.
\end{proposition}

For a proofs, see \cite{S243}.

\begin{example} Let $A$ be a commutative $k$-algebra, $k$ algebraically closed and $A$ finitely generated. Let $M=\oplus_{i=1}^r A/\mathfrak m_i$ where $\mathfrak m_i,\ 1\leq i\leq r,$ are maximal ideals. Then by the versal properties it follows that $\hat O^A_M\simeq\oplus_{i=1}^r\hat A_{\mathfrak m_i}.$ 
\end{example}

\subsection{The Ring of Local Functions}

All the results in this subsection  is proved in \cite{S243}.

\begin{definition} Let $M=\oplus_{i=1}^r M_i$ be a sum of right $A$-modules. Then the ring $O^A_M$ of functions locally defined at $M$ is defined as the sub-algebra of $\hat O^A_M$ generated by  $\rho^A_M(A)$ (over $k^r$), together with $\rho^A_M(a)^{-1}$ whenever $\rho^A_M(a)$ is a unit.
\end{definition}

\begin{proposition} The $k$-algebra $O^A_M$ is augmented over $\oplus_{i=1}^r\enm_k(M_i),$ and $$O^{O^A_M}_M\simeq O^A_M.$$
\end{proposition}

Consider a ring homomorphism $\phi:A\rightarrow B$ between two associative rings. Then every $B$-module $M$ is also an $A$-module via $\phi.$ More precisely, if $M$ is defined as $B$-module by the structure morphism $\eta^B_M,$ then $M$ is defined as an $A$-module by the structure morphism $\eta^A_M=\eta^B_M\circ\phi.$ We follow the usual notation, and call the $B$ module $M$ considered as $A$-module via $\phi$, $M_{\phi}^c.$

\begin{lemma} Let $B$ be a $k$-algebra and $M=\oplus_{i=1}^rM_i$ a sum of right $B$-modules.
Let $\phi:A\rightarrow B$ be an algebra homomorphism so that $M_{\phi}^c=\oplus_{i=1}^r (M_i)^c_\phi.$ Then $\phi$ induces a morphism of augmented algebras $\phi^\ast:O^A_{M_{\phi}^c}\rightarrow O^B_M$.
\end{lemma}

When $A$ is a commutative ring, every prime ideal $\mathfrak p\subset A$ is the contraction of the maximal ideal $\mathfrak p A_{\mathfrak p}\subset A_{\mathfrak p}$ where $A_{\mathfrak p}$ is a local ring. We generalize this to associative rings.

\begin{definition} A right $A$ module $P_A$ is called aprime if there is a morphism $f:A\rightarrow B$ and a simple right $B$-module $M_B$ such that $P_A=(M_B)^c_f.$
\end{definition}

\begin{lemma} The contraction of an aprime module is aprime.
\end{lemma}

\begin{proof} Follows directly by composition of homomorphisms.
\end{proof}

We now have all we need to give the definition of associative schemes.

\begin{definition} We let $\operatorname{aSpec}(A)$ denote the set of right aprime $A$-modules. For $f\in A$ we  put $D(f)=\{P\in\operatorname{aSpec}(A)|\eta^A_P(f)\text{ is  injective}\}$ and we give $\operatorname{aSpec}(A)$ the topology generated by the sub-basis $\{D(f)\}_{f\in A}.$
\end{definition}

\begin{lemma}\label{continlemma}
Let $\phi:A\rightarrow B$ be a ring homomorphism. Then the induced map $\aspec(\phi):\aspec B\rightarrow\aspec A$ given by $\aspec(\phi)(P)=P^\phi_A$ is continuous.
\end{lemma}

\begin{proof} $\aspec(\phi)^{-1}(D(f))=D(\phi(f)).$
\end{proof}

\begin{definition}\label{topsheafDef} We define a sheaf of rings $\mathcal O_X$ on $X=\operatorname{aSpec}A$ by:

(i) First, define the presheaf $$O_X(U)=\underset{\underset{M,U}\leftarrow}\lim\ O^A_M$$ where the limit is taken over all finite sums $M=\oplus_{i=1}^rM_i$ of simple right $A$-modules $M_i\in U.$ 

(ii) Then, define the sheaf $$\mathcal O_X(U)=\underset{\underset{V \subsetneqq U}\leftarrow}\lim O_X(V).$$
\end{definition}

For an explanation on how this defines a sheaf of associative rings, see \cite{S241}.

\begin{definition} For an associative unital ring $A$ we define $\operatorname{aSpec A}$ as the ringed space $(\operatorname{aSpec} A, \mathcal O_{\operatorname{aSpec} A}).$
\end{definition}

\begin{corollary} Let $\phi:A\rightarrow B$ be homomorphism of associative rings. Then there is a natural morphism $\operatorname{aSpec}(\phi):\operatorname{aSpec}B\rightarrow\operatorname{aSpec}A$
\end{corollary}

\begin{proof} The diagram $$\xymatrix{B\ar[drr]_{\eta^B_P}\ar[r]^\phi&A\ar[dr]^{\eta^A_P}\ar[r]^-\rho&A_P\ar[d]^{\eta^{A_P}_P}\\&&\enm_k(P)}$$ gives the map on points sending the aprime $A$-module $P$ to the aprime $B$-module $P.$ The universal property of $A_P$ gives that there is a unique morphism $B_P\rightarrow A_P$ which gives the corresponding morphism of pointed rings, thus a morphism of pre-sheaves, then of sheaves.
\end{proof}

\begin{definition} Let $\tilde P=\{P_1,...,P_r\}\subseteq X=\operatorname{aSpec}A$ be a finite set. We define the stalk of any presheaf $F_X$ on $X$ in $\tilde P$ as the inductive limit $$F_{X,\tilde P}=\underset{\underset{\tilde P\subseteq U}\rightarrow}\lim\ F_X(U).$$ 
\end{definition}

\begin{lemma} The stalk in $\tilde P$ of the sheaf $\mathcal O_X$ is isomorphic to the stalk in $\tilde P$ of the presheaf $O_X.$
\end{lemma}

\begin{proof} Inductive limits is a contravariant functor. Thus the morphism $O_X\rightarrow\mathcal O_X$ gives a morphism $\mathcal O_{X,\tilde P}\rightarrow O_{X,\tilde P}.$ Also, for each open $U\supseteq \tilde P$ we have by the same reason a homomorphism $O_{X}(U)\rightarrow \mathcal O_{X,\tilde P}$ giving a homomorphism $O_{X,\tilde P}\rightarrow\mathcal O_{X,\tilde P}.$ By uniqueness, these morphisms are inverses.
\end{proof}

\begin{corollary} If $\tilde P=\{P_i\}_{i=1}^r\subseteq X$ is a finite set of aprime right $A$-modules,  then $\mathcal O_{X,\tilde P}\simeq O^A_{P},\ P=\oplus_{i=1}^r P_i.$ 
\end{corollary}

\begin{example} When $A$ is a finitely generated domain over an algebraically closed field $k$ we have $\operatorname{aSpec}A\simeq\spec A.$
\end{example}

\begin{corollary} Let $A$ be a finite dimensional $\Bbbk$-algebra, $\Bbbk$ algebraically closed. Then the homomorphism  $$\eta:A\rightarrow O^A(\simp A)$$ where $\simp A$ is the set of all simple (right) $A$-modules, is an isomorphism.
\end{corollary}

\begin{proof} When $A$ is finite dimensional $\hat O^A(\tilde M)\simeq O^A(\tilde M)$ for any set of simple modules $\tilde M.$ Also, because $\Bbbk$ is algebraically closed, $\mathcal O^A(\operatorname{aSpec}A)\simeq\mathcal O^A(\simp A).$ 
\end{proof}

\begin{definition}\label{Schemedef} An associative scheme is a topological space with a  sheaf of associative rings $(X,\mathcal O_X)$ which has a cover of open affine subsets $X=\cup_{i\in I}\operatorname{aSpec}(A_i)$ where each $A_i$ is an associative ring. A morphism of associative schemes is  a morphism of ringed spaces. The category of associative schemes over $k$ is denoted $\cat{aSch}_k.$
\end{definition}

\begin{definition} An associative variety over an algebraically closed field $\Bbbk$ is an irreducible associative scheme over $\Bbbk.$
\end{definition}

\end{document}